\documentclass[12pt]{amsart}
\usepackage[utf8]{inputenc}
\usepackage[total={6.5in,9in},top=1in, left=1in, right=1in, bottom=1in]{geometry}
\usepackage{fancyhdr}
\usepackage{libertine}
\usepackage[T1]{fontenc}
\usepackage{color,amsmath,amssymb,mathtools,graphicx}
\usepackage{bbold}
\usepackage{epsfig,fancyhdr,xcolor}
\usepackage{dsfont} 
\usepackage{graphicx,amsmath,mathtools,amssymb,float}
\usepackage{bm}
\usepackage{epsfig,color,fancyhdr,setspace}
\usepackage{epstopdf}
\usepackage{import}
\usepackage{lineno}
\usepackage{stackrel,dsfont}
\usepackage[allcolors = blue,colorlinks]{hyperref}
\usepackage[abs]{overpic}
\usepackage[center]{caption}
\usepackage{subcaption}
\usepackage{tikz-cd}
\usepackage{enumitem}
\usepackage{marvosym}
\usepackage[normalem]{ulem}

\newtheorem{theorem}{Theorem}[section]

\newtheorem{definition}[theorem]{Definition}
\newtheorem{example}[theorem]{Example}
\newtheorem{proposition}[theorem]{Proposition}
\newtheorem{remark}[theorem]{Remark}
\newtheorem{corollary}[theorem]{Corollary}
\newtheorem{conjecture}[theorem]{Conjecture}
\newtheorem{openproblem}[theorem]{Open Problem}

\newtheorem*{conjecture*}{Unimodality Conjecture}

\title{
Exploring unimodality of the plucking polynomial with delay function}

\author{Dionne Ibarra}
\address{School of Mathematics, Monash University, Australia}
\email{{\rm \textcolor{blue}{dionne.ibarra@monash.edu}}}

\author{Alex Landry}
\address{James Madison High School, VA, USA}
\email{{\rm \textcolor{blue}{alexc.landry@gmail.com}}}

\author{Gabriel Montoya-Vega}
\address{Department of Mathematics, The Graduate Center CUNY, NY, USA, and \newline \indent Department of Mathematics, University of Puerto Rico at R\'io Piedras, San Juan, PR}
\email{{\rm \textcolor{blue}{gabrielmontoyavega@gmail.com}}}

\author{J\'{o}zef H. Przytycki}
\address{Department of Mathematics, The George Washington University, Washington DC, USA, and \newline \indent Department of Mathematics, University of Gda\'{n}sk, Gda\'{n}sk, Poland}
\email{{\rm \textcolor{blue}{przytyck@gwu.edu}}}


\begin{document}

\begin{abstract}
    The plucking polynomial is an invariant of rooted trees with connections to knot theory. The polynomial was constructed in 2014 as a tool to analyze lattice crossings after taking the quotient by the Kauffman bracket skein relations. In this paper we study the plucking polynomial and the plucking polynomial with delay function. We present a formula for the plucking polynomial of hedgehog rooted trees and explore the unimodality of this polynomial. In particular, we consider an anti-unimodal delay function and a delay function with a specific image set. Furthermore, we present a number of interesting examples and make some speculations on the unimodality of plucking polynomials with delay functions of hedgehog rooted trees.
\end{abstract}
\maketitle

\tableofcontents

\section{Introduction}
The plucking polynomial is an invariant of rooted trees that enjoys connections with the mathematical theory of knots.  The fourth author introduced the plucking polynomial (originally called the $q$-polynomial of rooted trees) after completing a work with M. D{\c a}bkowski and C. Li concerning the Kauffman bracket skein module of  lattice crossings \cite{DLP, Prz}. The plucking polynomial's connections to lattice crossings is explored in \cite{DP}. A further study of the properties of the plucking polynomial   can be found in \cite{CMPWY1, CMPWY2, CMPWY3}. 

\ 

The plucking polynomial is a polynomial defined recursively from a plane rooted tree by ``plucking'' the leaves of the tree. This article is organized as follows. In this section we first define a rooted tree and introduce the standard version of the plucking polynomial together with a calculation example. Then we define the plucking polynomial with delay function. In Section \ref{pluckinghh} we define a hedgehog rooted tree and show a calculation of its plucking polynomial. There we present a formula for the calculation of the plucking polynomial with delay function and we study its unimodality under certain conditions. Section \ref{calculationsExamples} consists of some examples and a generalization of the structure of the polynomial for a specific hedgehog rooted tree. Finally, in Section \ref{FutureWork} we present some questions and speculations that may guide future directions of related investigation.

\begin{definition}\label{rootedtree}
A \textbf{rooted tree} is a graph that is connected, does not contain any cycles (acyclic), and one vertex is designated as the root. A plane rooted tree is a rooted tree embedded in the plane.  Figure \ref{PLUCKING:examplerootedtree} shows examples of plane rooted trees.
\end{definition}	

\begin{figure}[ht]
\centering
\begin{subfigure}{.45\textwidth}
\centering
$\vcenter{\hbox{
\begin{overpic}[scale=.8]{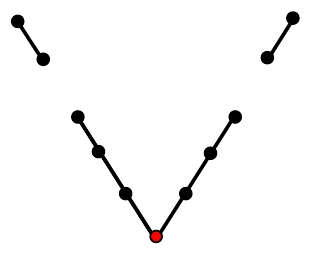}
\put(24, 56){\fontsize{9}{9}$\bullet$}
\put(21, 61){\fontsize{9}{9}$\bullet$}
\put(18, 66){\fontsize{9}{9}$\bullet$}
\put(97, 66){\fontsize{9}{9}$\bullet$}
\put(91, 56){\fontsize{9}{9}$\bullet$}
\put(94, 61){\fontsize{9}{9}$\bullet$}
\put(80, 20){$a$ edges}
\put(5, 20){$b$ edges}
\put(65,1){root}
\put(-2,80){$v_2$}
\put(114,81){$v_1$}
\end{overpic} }} $
\caption{$T_{b, a}$} \label{PLUCKING:exampleTab}
\end{subfigure}
\begin{subfigure}{.45\textwidth}
\centering
$\vcenter{\hbox{
\begin{overpic}[scale=.8]{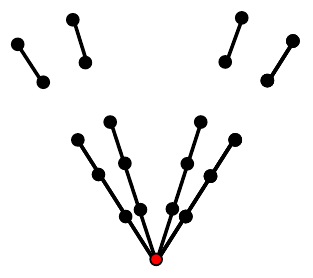}
\put(24, 56){\fontsize{9}{9}$\bullet$}
\put(21, 61){\fontsize{9}{9}$\bullet$}
\put(18, 66){\fontsize{9}{9}$\bullet$}
\put(97, 66){\fontsize{9}{9}$\bullet$}
\put(91, 56){\fontsize{9}{9}$\bullet$}
\put(94, 61){\fontsize{9}{9}$\bullet$}
\put(65,1){root}
\put(-2,80){$v_k$}
\put(34,93){$v_{k-1}$}
\put(94,90){$v_2$}
\put(114,80){$v_1$}
\put(37.5, 62){\fontsize{9}{9}$\bullet$}
\put(35.5, 67){\fontsize{9}{9}$\bullet$}
\put(33.5, 72){\fontsize{9}{9}$\bullet$}
\put(82, 72){\fontsize{9}{9}$\bullet$}
\put(78, 62){\fontsize{9}{9}$\bullet$}
\put(80, 67){\fontsize{9}{9}$\bullet$}
\end{overpic} }} $
\caption{$T_{a_k, a_{k-1}, \dots, a_1}$} \label{PLUCKING:exampleT1n}
\end{subfigure}
\caption{Example of plane rooted trees: (a) $T_{b,a}$ is a rooted tree with 2 branches of length $b$ and $a$, and each branch has one leaf, $v_2$ and $v_1$, respectively. (b) $T_{a_k, a_{k-1}, \dots, a_1}$ is a rooted tree with $k$ branches of length $a_i$ for $1\leq i \leq k$, and each branch has one leaf.} \label{PLUCKING:examplerootedtree}
\end{figure}
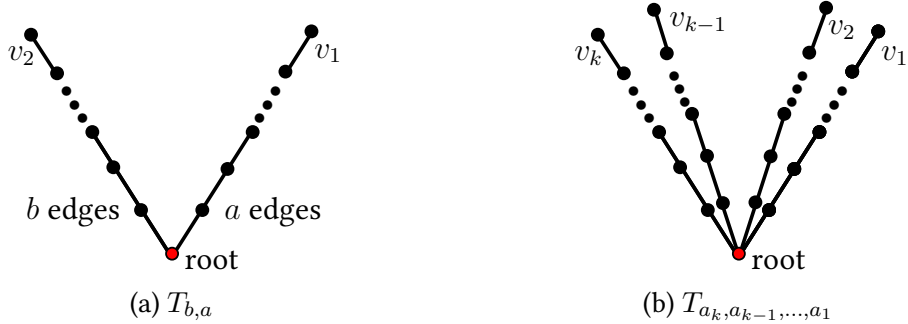

\begin{definition}\label{pluckingpolynomial}
Let $T$ be a plane rooted tree and let $|E(T)|$ denote the number of edges of $T$. The \textbf{plucking polynomial} $Q(T, q)$ (or succinctly $Q(T)$)  is the unique polynomial obtained from $T$ by the initial condition and recursive formula shown below:

\ 

\begin{enumerate}
\item If $T$ is the one vertex tree, i.e. $|E(T)|=0$, then $Q(T, q) =1$.

\

\item If $|E(T)|>0$, then 
$$ Q(T,q) = \sum_{v \in L(T)} q^{r(T,v)} Q(T-v, q),$$
where the sum is taken over the set $L(T)$ of all leaves of $T$, that is, vertices of degree different from the root; $r(T,v)$ is the number of edges of $T$ to the right of the unique path connecting  $v$ to the root; $T-v$ is a rooted tree obtained by removing the leaf $v$ from $T$ along with the edge directly connected to $v$. 
\end{enumerate}
\end{definition}

\begin{example}
Consider the following plane rooted tree $T$.
$$ T = \vcenter{\hbox{
\begin{overpic}[scale = .8]{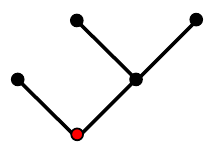}
\put(-2,18){$v_3$}
\put(20,41){$v_2$}
\put(76,41){$v_1$}
\end{overpic} }} $$

The unique path connecting each leaf $v_i$ to the root for $i=1,2,3$ is shown in Figure \ref{PLUCKING:exampleuniquepath}. We obtain $r(T, v_1)= 0$, $r(T, v_2)= 1$, and $r(T, v_3)= 3$.

\begin{figure}[H]
\centering
$ \vcenter{\hbox{
\begin{overpic}[scale = .8]{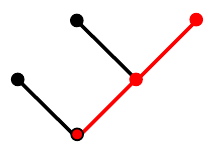}
\put(-2,18){$v_3$}
\put(20,41){$v_2$}
\put(76,41){$v_1$}
\end{overpic} }} \ \ \ \vcenter{\hbox{
\begin{overpic}[scale = .8]{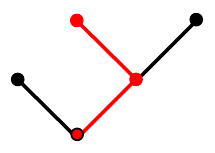}
\put(-2,18){$v_3$}
\put(20,41){$v_2$}
\put(76,41){$v_1$}
\end{overpic} }} \ \ \ \vcenter{\hbox{
\begin{overpic}[scale = .8]{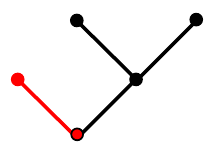}
\put(-2,18){$v_3$}
\put(20,41){$v_2$}
\put(76,41){$v_1$}
\end{overpic} }} $
\caption{The red segments highlight the unique path from $v_i$ to the root for $i = 1, 2, 3$.} \label{PLUCKING:exampleuniquepath}
\end{figure}

Then $Q(T,q) = q^0 Q(T-v_1)+q^1 Q(T-v_2)+q^3 Q(T-v_3)$, where 

$$ T-v_1=\vcenter{\hbox{
\begin{overpic}[scale = .8]{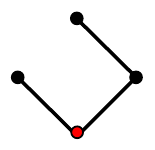}
\put(-2,18){$v_3$}
\put(20,41){$v_2$}
\end{overpic} }}, \ 
T-v_2= \vcenter{\hbox{
\begin{overpic}[scale = .8]{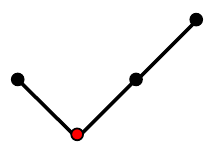}
\put(-2,18){$v_3$}
\put(76,41){$v_1$}
\end{overpic} }}, \ 
T-v_3=\vcenter{\hbox{
\begin{overpic}[scale = .8]{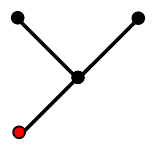}
\put(-2,41){$v_2$}
\put(54,41){$v_1$}
\end{overpic} }}.$$

Observe that $Q(T-v_1) = Q(T-v_2)$, so we only need to compute $Q(T-v_1)$ and $Q(T-v_3)$.

\begin{eqnarray*}
Q(T-v_1) & = & 
q^0 Q \left(\vcenter{\hbox{
\begin{overpic}[scale = .5]{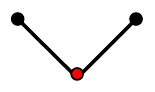}
\put(-1,8){\fontsize{8}{8}$v_2'$}
\put(29,8){\fontsize{8}{8}$v_1'$}
\end{overpic} }} \right)+ 
q^2 Q \left( \vcenter{\hbox{
\begin{overpic}[scale = .5]{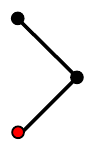}
\put(10,27){\fontsize{8}{8}$v_1'$}
\end{overpic} }} \right) \\
&=& \left(
 q^0  Q\left( \vcenter{\hbox{\includegraphics[scale = .5]{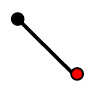}}} \right) + 
q^1 Q \left( \vcenter{\hbox{\includegraphics[scale = .5]{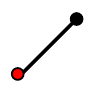}}} \right) \right) + 
q^2 \left(
q^0 Q \left(\vcenter{\hbox{\includegraphics[scale = .5]{RTex6.pdf}}} \right) \right)\\
&=& Q(\vcenter{\hbox{\includegraphics[scale = .5]{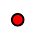}}}) + q Q( \vcenter{\hbox{\includegraphics[scale = .5]{RTex8.pdf}}}) + q^2 Q(
 \vcenter{\hbox{\includegraphics[scale = .5]{RTex8.pdf}}}) = 1+q+q^2,
\end{eqnarray*}
and 
\begin{eqnarray*}
Q(T-v_3) & = & q^0 Q \left(\vcenter{\hbox{
\begin{overpic}[scale = .5]{RTex5}
\put(10,27){\fontsize{8}{8}$v_1'$}
\end{overpic} }} \right)+ 
q^1 Q \left( \vcenter{\hbox{
\begin{overpic}[scale = .5]{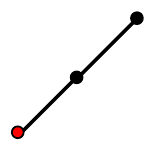}
\put(20,30){\fontsize{8}{8}$v_1'$}
\end{overpic} }} \right)
= (1+q) Q \left(\vcenter{\hbox{
\begin{overpic}[scale = .5]{RTex5}
\put(10,27){\fontsize{8}{8}$v_1'$}
\end{overpic} }} \right) \\
&=& (1+q) q^0 Q \left( \vcenter{\hbox{\includegraphics[scale = .5]{RTex6.pdf}}} \right) \\
&=& (1+q) q^0 Q(\vcenter{\hbox{\includegraphics[scale = .5]{RTex8.pdf}}}) = 1+q.
\end{eqnarray*}
Therefore, $Q(T,q) = (1+q)(1+q+q^2)+ q^3 (1+q) = (1+q)(1+q+q^2+q^3).$

\end{example}

\begin{theorem}\label{independenceofembedding} \cite{Prz}
The plucking polynomial $Q(T)$ of a rooted tree $T$ does not depend on the plane embedding. $Q(T)$ is therefore an invariant of rooted trees.    
\end{theorem}
\begin{proof}
    See Theorem 2.2 and Corollary 2.3 in \cite{Prz}.    
\end{proof}

Notice that in the previous example we calculated the plucking polynomial of a special rooted tree which was shown in Figure \ref{PLUCKING:exampleTab}. Definition \ref{PLUCKING:definitionqinteger} introduces the notions of quantum integers, q-factorials, and Gaussian polynomials. Thus we can write $Q(T- v_1) = Q(T_{1,2}) = [3]_q$ and $Q(T,q) = [2]_q[4]_q$.

\begin{definition} \label{PLUCKING:definitionqinteger}
$[n]_q = 1 +q + \cdots + q^{n-1}$ is called a \textbf{quantum integer}, $[n]_{q} ! = \prod\limits_{k=1}^{n} [k]_{q}$ is called a $\boldsymbol{q}$\textbf{-factorial},
and 
$$\binom{n}{i, n-i}_q = \frac{[n]_q!}{[i]_q![n-i]_q!}$$
is called a \textbf{Gaussian polynomial}. By convention $[0]_q!=1$, $\binom{n}{0}_q =1 = \binom{n}{n}_q$, and $\binom{n}{-1} = 0$. 
\end{definition}

\begin{example}
Consider the rooted tree with two long branches of length $b$ and $a$ respectively, shown in Figure \ref{PLUCKING:exampleTab}. One natural question that arises is how many different ways there are to “pluck” the tree $T_{b,a}$ one leaf at a time. In the case where the plane is commutative, the left branch should be plucked $b$ times and the right branch $a$ times, so the answer in this case is given by ${a+b \choose a}$. On the other hand, when the plane is $q$-commutative, i.e. $yx=qxy$, 
For a detailed discussion see \cite{Prz, PBIMW}.
\end{example}

\begin{definition}\label{delayfunction}
The plucking polynomial $Q(T,f)$ of a plane rooted tree $T$ with a \textbf{delay function} $f: L(T) \to \mathbb{Z}^+$ is given by the initial conditions $Q(\vcenter{\hbox{\includegraphics[scale = .5]{RTex8.pdf}}}) =1$, $Q(T,f)=0$ if $f(w)\geq 2$ for any leaf $w$, and the recursive relation:
$$Q(T,f)= \sum_{v \ \in \ L_1(T)} q^{r(T,v)}Q(T-v,f_v),$$
where $L_1(T)=f^{-1}(1)\in L(T)$, and $f_v: L(T-v) \to \mathbb{Z}^+$ is given by

$$ f_v(u)=\left\{
                \begin{array}{ll}
                max\{1,f(u)-1\} & \mbox{ when $u\in L(T)$},\\
                 \ \ \ \ \ \ \ \ \ \ \ \ \ 1 &  \mbox{ when $u\notin L(T), u\in L(T-v)$}.
                \end{array}
                \right.$$
                
\end{definition}

\begin{example}
    Consider the following plane rooted tree $T$ with delay function:
$$ (T, f) = \vcenter{\hbox{
\begin{overpic}[scale = .8]{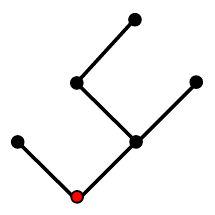}
\put(3,33){$2$}
\put(50,80){$3$}
\put(74,56){$1$}
\end{overpic} }}. $$
Then,
\begin{eqnarray*}
   Q(T, f) =  Q \left( \vcenter{\hbox{
\begin{overpic}[scale = .8]{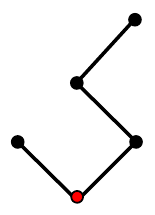}
\put(3,33){$1$}
\put(50,80){$2$}
\end{overpic} }} \right) = q^{3} Q \left(\vcenter{\hbox{
\begin{overpic}[scale = .8]{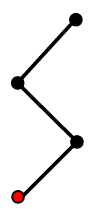}
\put(30,80){$1$}
\end{overpic} }}  \right) = q^3.
\end{eqnarray*}
\end{example}

\section{Plucking Polynomial of the Hedgehog Rooted Tree}\label{pluckinghh}	
A hedgehog graph is a rooted tree in which every edge is a leaf. A hedgehog graph can be written as $T_{1, 1, ..., 1}$, a root with $n$ protruding edges; see Figure \ref{fig:hedgehog}.


	\begin{figure}[h]
	    \centering
	    $\vcenter{\hbox{
\begin{overpic}[scale = .8]{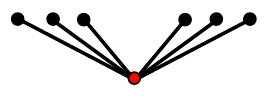}
\put(49,23){\fontsize{7}{7}$\bullet$}
\put(40,23){\fontsize{7}{7}$\bullet$}
\put(58,23){\fontsize{7}{7}$\bullet$}
\end{overpic} }}$
	    \caption{A hedgehog rooted tree.}
	    \label{fig:hedgehog}
	\end{figure}

\begin{example}
    Consider the hedgehog rooted tree $T_{1, 1, \cdots, 1}$ with $n$ leaves as shown in Figure \ref{fig:hedgehog}. Since $T_{1, 1, \cdots, 1}-v_1 =T_{1, 1, \cdots, 1}-v_2$ for all $v_1, v_2 \in L(T_{1, 1, \cdots, 1})$, then $Q(T_{1, 1, \cdots, 1}) = [n]_q Q(T_{1, 1, \cdots, 1}-v)$. By induction we have
    $$Q(T_{1, 1, \cdots, 1}) = [n]_q!.$$
\end{example}

\begin{example}\label{exampledelaydefi}
We use Definition \ref{delayfunction} to calculate the plucking polynomial of the hedgehog rooted tree shown in Figure \ref{fig:32123example}.

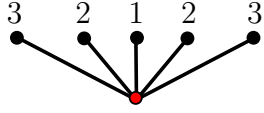
\begin{figure}[H]
	\centering	$$\vcenter{\hbox{\begin{overpic}[scale = .8]{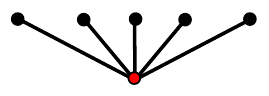}
\put(3,35){$3$}
\put(29,35){$2$}
\put(49,35){$1$}
\put(69,35){$2$}
\put(94,35){$3$}
\end{overpic}}}$$	\caption{Hedgehog rooted tree with delay values $2$ and $3$.}
\label{fig:32123example}
\end{figure}    

 \begin{align*}
Q \left(  \vcenter{\hbox{\begin{overpic}[scale = .8]{Tree5hedge}
\put(3,35){$3$}
\put(29,35){$2$}
\put(49,35){$1$}
\put(69,35){$2$}
\put(94,35){$3$}
\end{overpic}}} \right)  & =
q^{2} \cdot Q\left(  \vcenter{\hbox{\begin{overpic}[scale = .8]{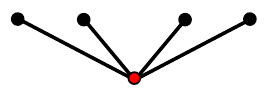}
\put(3,35){$2$}
\put(29,35){$1$}
\put(69,35){$1$}
\put(94,35){$2$}
\end{overpic}}} \right) & \\\\ 
\hspace{3cm} &= q^{2} \left[  q\cdot  Q \left(  \vcenter{\hbox{\begin{overpic}[scale = .8]{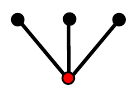}
\put(3,35){$1$}
\put(24,35){$1$}
\put(44,35){$1$}
\end{overpic}}}\right)+q^{2} \cdot Q \left(  \vcenter{\hbox{\begin{overpic}[scale = .8]{Tree3hedge}
\put(3,35){$1$}
\put(24,35){$1$}
\put(44,35){$1$}
\end{overpic}}}\right) \right] &  \\\\
   \hspace{3cm} & = (q^{3} +q^4)\cdot  Q \left(  \vcenter{\hbox{\begin{overpic}[scale = .8]{Tree3hedge}
\put(3,35){$1$}
\put(24,35){$1$}
\put(44,35){$1$}
\end{overpic}}}\right) \\\
    \hspace{3cm} & = \left[ 3 \right]_{q}! \cdot (q^{3}+q^{4}) \ \ \  = \ \ \  \left[ 3 \right]_{q}\cdot  \left[ 2 \right]_{q} \cdot  \left[ 1 \right]_{q}  \cdot (q^{3}+q^{4}) & \\\
     \hspace{3cm} & = (1+q+q^{2})\cdot (1+q) \cdot (q^{3}+q^{4}) \ \ \ = \ \ \ q^{7}+3q^{6}+4q^{5}+3q^{4}+q^{3} &\\\
     \hspace{3cm} & = q^{3}  \left[ 3 \right]_{q}  \left[ 2 \right]^{2}_{q}.
     \end{align*}
\end{example}
	
In the following two subsections, we analyze special cases of the plucking polynomial of a hedgehog rooted tree with delay function. First we prove that the plucking polynomial is unimodal when the delay function $f$ is anti-unimodal in Subsection \ref{unimodalitysection}. Then we consider the specific case of a delay function whose image is the set $\{1, 2\}$ in Subsection \ref{specialdelaysection}.
 
\subsection{Unimodality}\label{unimodalitysection} We say that that the sequence $a_1, a_2, \cdots a_n$ is unimodal if there exists $ 1\leq j \leq n$ such that  
$$a_1 \leq a_2 \leq ... \leq a_j \geq ... \geq a_n.$$

The polynomial $P$ is called unimodal if its coefficients form a unimodal sequence.
Consider a plane rooted tree $T$ with delay function $f: L(T) \to \mathbb{Z}^+$. We say that this function is anti-unimodal if $-f$ is unimodal. \\

In this subsection we prove that in the case of $f$ anti-unimodal, $Q(T_{1,...,1},f)$ is the product of quantum integers therefore $Q(T_{1,...,1},f)$ is unimodal. 

	
\

Let $T$ be the plane rooted tree $T_{1, 1, \cdots, 1}$, that is, a rooted tree embedded in the plane with a root and $n$ other vertices which are all leaves. Denote the leaves (from left to right) by $v_1,...,v_n$. Let the delay function $f$ be defined by  $f(v_i)=a_i$. We use the notation $|f^{-1}(i)|=f_i$, and $M=max(a_1,a_n)$. Assume that the function $f$ is anti-unimodal, that is, for some $j$ we have that
$$a_1 \geq a_2 \geq ... \geq a_j \leq ... \leq a_n.$$
and that $f(v_j)=a_j=1$; otherwise $Q(T,f)$ would be  equal to zero. The last piece of notation we need is the following. For $i>1$, let $f_i=f_i^- + f_i^+$, where $f_i^+$ is the number of vertices in $f^{-1}(i)$ which are to the right of $v_j$ in the tree. The precise formulation is given in Theorem \ref{antiunimodalformula}.

\begin{theorem}\label{antiunimodalformula} Let $T = T_{1, 1, \cdots, 1}$ be a hedgehog rooted tree. Let $f$ be an anti-unimodal delay function as described above. Then
$$Q(T,f)= q^{\sum_{i=2}^{a_n}(i-1)f_i^+ }\prod_{j=1}^n\bigg[(\sum_{i=1}^jf_i)-j+1\bigg]_q.$$ 
In particular, $Q(T,f)$ is unimodal as it is the product of quantum integers.
\end{theorem}

\begin{proof}
The proof proceeds by induction on $n$ - the number of leaves of $T$. Recall the recursive formula from the definition of the plucking polynomial:
$$ Q(T,f) = \sum_{v \in L(T)} q^{r(T,v)} Q(T-v, f_v).$$
In this case, the rooted tree with delay function $(T-v, f_{v})$ is independent of the vertex $v$. Therefore, it would be again a hedgehog with anti-unimodal function. Thus,
$$Q(T,f) = \left(  \sum_{v \in L_{1}(T)} q^{r(T,v)} \right) \cdot Q(T-v, f_{v}),$$
for arbitrarily chosen $v$ in $L_{1}(T)$. Let us analyze how we find $q^{r(T,v)}$ for any $v \in L_{1} $. Consider the vertex $v'$ in $L_{1}$ and suppose it is the most to the right among all vertices. Then $r(T,v')$ is the number of vertices of $L$ to the right of $v'$. In our notation:
$$r(T,v')=|f^{-1}(2)|+|f^{-1}(3)|+ \cdots +|f^{-1}(a_{n})|.$$
Then,
\begin{eqnarray*}
  \displaystyle\sum_{v \in L_{1}(T)} q^{r(T,v)} &=& q^{r(T,v)} + qq^{r(T,v')} + q^{2}q^{r(T,v')}+ \cdots + q^{|f^{-1}(1)-1|}q^{r(T,v')} \\ &=& (1+q+q^{2}+q^{3}+ \cdots + q^{|f^{-1}(1)|-1})\cdot q^{r(T,v')} \\ &=& \left[ f^{-1}(1)  \right]_{q} \cdot q^{r(T,v')}.
\end{eqnarray*}

Therefore,
$$Q(T,f)=\left( \displaystyle\sum_{v \in L_{1}(T)} q^{r(T,v)}\right) \cdot Q(T-v,f_{v})=\left[ f^{-1}(1)  \right]_{q} \cdot q^{r(T,v)} \cdot Q(T-v',f_{v'}).$$
Observe

 that for $f\equiv 1$ we get $Q(T)=[n]_q!$ as needed. Moreover, notice that if $j > \sum\limits_{i=1}^j f_i$ then the sum is equal to zero.
\end{proof}

\begin{example}
Consider again the hedgehog rooted tree from Example \ref{exampledelaydefi} illustrated in Figure \ref{fig:32123example}. Let us calculate its plucking polynomial using the formula from Theorem \ref{antiunimodalformula}.

\

From the figure we can determine all the variables for the formula. First we have that $max(a_1,a_n)=max(3,3)=3$. Then we see that $f(v_1)=a_1=3$, $f(v_2)=a_2=2$, $f(v_3)=a_3=1=a_j$, $f(v_4)=a_4=2$, and $f(v_5)=a_5=3$. Recall that $f_i=f^{-}_i+f^{+}_i$ for $i > 0$, where $f_i=|f^{-1}(i)|$. This means that we need to consider the cases for $i=2, 3, 4, 5$.
\begin{itemize}
    \item For $i=2$, $f_2=f^{-}_2+f^{+}_2=|f^{-1}(2)|=2$.
    \item For $i=3$, $f_3=f^{-}_3+f^{+}_3=1+1=2$.
    \item For $i=4$, we have $f_4=0$.
    \item For $i=5$, we have $f_5=0$.
\end{itemize}
Moreover, observe that
$$\displaystyle \sum _{i=2}^{a_n}(i-1)f_{i}^{+}= \displaystyle \sum_{i=2}^{3}(i-1)f_{i}^{+}= (2-1)f_{2}^{+}+(3-1)f_{3}^{+}=(1)(1)+(2)(1)=3, $$
which means that from here we get the factor $q^{3}$. Now observe that

\begin{eqnarray*}
    \displaystyle \prod_{j=1}^{5} \left[ \left( \displaystyle \sum_{i=1}^{j} f_i  \right) -j+1 \right]_{q} &=&  \left[  \left( \displaystyle\sum_{i=1}^{1}f_{i}  \right)-1+1   \right]_{q} 
\left[ \left( \displaystyle\sum_{i=1}^{2}f_{i}  \right) -1   \right]_{q} 
\left[   \left( \displaystyle\sum_{i=1}^{3}f_{i}  \right)-2  \right]_{q} \\&&
\left[  \left( \displaystyle\sum_{i=1}^{4}f_{i}  \right)-3    \right]_{q}
\left[ \left( \displaystyle\sum_{i=1}^{5}f_{i}  \right)-4    \right]_{q} \\&=& \left[  f_1  \right]_{q} \left[ f_1+f_2-1    \right]_{q}  \left[  f_1+f_2+f_3-2  \right]_{q}  \left[ f_1+f_2+f_3+f_4-3   \right]_{q} \\&& \left[f_1+f_2+f_3+f_4+f_5-4    \right]_{q}\\&=& \left[ 1  \right]_{q} \left[  2 \right]_{q}    \left[  3 \right]_{q} \left[  2 \right]_{q} \left[  1 \right]_{q} = (\left[ 2    \right]_{q})^{2} \left[ 3 \right]_{q}.
\end{eqnarray*}

Thus $Q(T, f)=q^{3} (\left[ 2    \right]_{q})^{2} \left[ 3 \right]_{q}$, which agrees with the result obtained in Example \ref{exampledelaydefi}.
\end{example}

\subsection{Special Delay Functions of the Hedgehog Rooted Tree}\label{specialdelaysection} Now we consider a delay function whose image is equal to the set $\{1, 2\}$. The precise formulation for computing the plucking polynomial in this case is given in Proposition \ref{pluckingdelay12}.	
	
\begin{proposition} \label{pluckingdelay12}
	Let $T = T_{1, 1, \cdots, 1}$ be the Hedgehog (star) rooted tree with n rays and let $f$ be a delay function with values $1$ or $2$. By ordering the leaves from right to left we let $\varepsilon_i = 1$ if $f$(i'th leaf) $=1$ and $\varepsilon_i = 0$ if $f$(i'th leaf)$=2$. Then $$Q(T,f)=p_n(q)[n-1]_q!,$$
	where $p_n(q)=\sum_{i=0}^{n-1} \varepsilon_i q^i$.
\end{proposition}


\begin{proof} Observe that $(T-v,f_v)$ for any leaf $v$ is the rooted tree $T_{1,1,...,1,1}$ with $n-1$ rays and trivial delay
function (that is $f_v\equiv 1$). Thus $Q(T-v,f_v)=[n-1]_q!$.
\end{proof}

It is still open whether the plucking polynomial of a rooted tree as in the previous proposition is always unimodal, though experimental calculations suggest this is the case. See Conjecture \ref{unimodconjecture}.

\
 
Proposition \ref{pluckingdelay12} suggests that it is important to analyze polynomials of the form $p_n(q)[n-1]_q!$, where $p_n(q)= \varepsilon_0 + \varepsilon_1 q+...+\varepsilon_{n-1}q^{n-1}$ and $\varepsilon_i= 1 \mbox{ or } 0$. 
This analysis is still open, however it is possible to show unimodality for the polynomial $p_{n+1} [n]_q$ where $p_{n+1}(q)=\varepsilon_0 + \varepsilon_1 q+ \varepsilon_2 q^2+...+ \varepsilon_n q^n$ and $\varepsilon_i = 0$ or $1$.

\begin{proposition}\label{Proposition 3.3} The polynomial $p_{n+1}(q) [n]_q$ is unimodal.
\end{proposition}
\begin{proof}
Observe that
$$(\varepsilon_0 + \varepsilon_1 q+ \varepsilon_2 q^2+...+ \varepsilon_n q^n)(1+q+q^2+...+q^{n-1})=$$
$$\varepsilon_0 + (\varepsilon_0 + \varepsilon_1)q + (\varepsilon_0 + \varepsilon_1 + \varepsilon_2)q^2+...+$$
$$(\varepsilon_0 + \varepsilon_1 + \varepsilon_2+...+\varepsilon_{n-1})q^{n-1} + (\varepsilon_1 + \varepsilon_2+...+\varepsilon_{n})q^n +$$
$$(\varepsilon_2+...+\varepsilon_{n})q^{n+1}+ (\varepsilon_3+...+\varepsilon_{n})q^{n+2}+...+ \varepsilon_{n})q^{2n-1}.$$
Therefore, if $p_{n+1}(q) [n]_q= \sum_{i=0}^{2n-1}{c_{i}q^{i}}$ then
$$c_0 \leq c_1 \leq ... \leq c_{n-1} \mbox{ and } c_n\geq c_{n+1} \geq ... \geq c_{2n-1}.$$
From this we conclude that $p_{n+1}(q) [n]_q$ is unimodal. Moreover, notice that
\[ c_n-c_{n-1}=\varepsilon_{n} - \varepsilon_{n-1}=
\left\{
\begin{array}{ll}
1 & \mbox{when $\varepsilon_{n}=1$, $\varepsilon_{0}=0$}, \\
0 & \mbox{when $\varepsilon_{n}=\varepsilon_{0}$}, \\
-1 & \mbox{when $\varepsilon_{n}=0$, $\varepsilon_{0}=1$}.
\end{array}
\right.
\]
\end{proof}

The proof of the previous theorem also leads to the result that $p_{n+1} [n]_q$ is strongly unimodal if and only if there are $i\leq j$ such that $\varepsilon_i=\varepsilon_{i+1}=\cdots=\varepsilon_{j}=1$.

\begin{corollary}\label{CoroPn} If $p_{n+1}(q)$ is symmetric then $p_{n+1}(q) [n]_q!$ is unimodal.
\end{corollary}
\begin{proof} By Proposition \ref{Proposition 3.3} 
$p_{n+1}(q) [n]_q$ is unimodal. Moreover, it is also symmetric as a product of symmetric polynomials. Thus, by a classical result\footnote{See, for instance, \cite{Win,Sta}.} that the product of symmetric unimodal polynomials is symmetric and unimodal, we  conclude that  $p_{n+1}(q) [n]_q!$ is unimodal.
\end{proof}

Observe that the same method we used in Proposition \ref{Proposition 3.3} can be applied to the polynomial $p_{n+2}(q)[n+1]_q[n]_q$ and succeed with some effort. However, our main focus is the polynomial $p_{n+1}(q) [n]_q!$. To this effect we have:
$$[n+1]_q[n]_q= 1+2q+3q^2+...+ nq^{n-1}+nq^n+ (n-1)q^{n+1}+...+ 2q^{2n-2}+ q^{2n-1}.$$

After some laborious calculation for 
$p_{n+2}(q)[n+1]_q[n]_q$ we obtain:

\begin{proposition}
The polynomial $p_{n+2}(q)[n+1]_q[n]_q$ is unimodal. Here  $p_{n+2}(q)$ is an arbitrary polynomial $p_{n+2}(q)= \varepsilon_0 + \varepsilon_1 q+..+ \varepsilon_{n+1}q^{n+1}$, where $\varepsilon_i = 0$ or $1$. \\
More precisely we can show that if $p_{n+2}(q)[n+1]_q[n]_q=\sum_{i=0}^{3n} c_iq^i$ then
$$c_0 \leq c_1 \leq ... \leq c_n \leq ... \leq c_{n+j} \geq ... \geq c_{2n} \geq c_{2n-1} \geq ... \geq c_{3n-1}\geq c_{3n},$$
where $j$ is the smallest number, $0\leq j \leq n$  such that $$c_{n+j+1}-c_{n+j}= -\varepsilon_{0}-\varepsilon_{1} - \varepsilon_{j} +\hat\varepsilon_{j+1}+ \varepsilon_{j+2}+...+ \varepsilon_{n} + \varepsilon_{n+1} \leq 0.$$
\end{proposition}

\begin{proof} we have
\begin{eqnarray*}
  p_{n+2}(q)[n+1]_q[n]_q &=& (\varepsilon_0 + \varepsilon_1 q+\cdots+ \varepsilon_{n+1}q^{n+1})\cdot \\&& (1+2q + 3q^2 +\cdots+nq^{n-1}+nq^n + (n-1)q^{n+1}\\&& +\cdots+2q^{2n-2}+q^{2n-1})\\&&= \varepsilon_0 + (2\varepsilon_0 + \varepsilon_1)q + (3\varepsilon_0 + 2\varepsilon_1 +\varepsilon_2)q^2\\&& +\cdots
+ (n\varepsilon_0 +(n-1)\varepsilon_1 +\cdots + \varepsilon_{n-1})q^{n-1}\\&& +(n\varepsilon_0 +n\varepsilon_1+ (n-1)\varepsilon_2+\cdots+\varepsilon_n)q^n\\&& +
((n-1)\varepsilon_0 +n\varepsilon_1 + n\varepsilon_2 +(n-1)\varepsilon_3 +\cdots +\varepsilon_{n+1})q^{n+1}\\&& +((n-2)\varepsilon_0 +(n-1)\varepsilon_1 + n\varepsilon_2 +(n)\varepsilon_3 +\cdots+2\varepsilon_{n+1})q^{n+2}+\cdots\\&&+ 2\varepsilon_0 + 3\varepsilon_1 +\cdots+n\varepsilon_{n-2}+n\varepsilon_{n-1}+ (n-1)\varepsilon_{n}+ (n-2)\varepsilon_{n+1})q^{2n-2}\\&&+\varepsilon_0 + 2\varepsilon_1 +\cdots+n\varepsilon_{n-1}+n\varepsilon_{n}+ (n-1)\varepsilon_{n}+ (n-1)\varepsilon_{n+1})q^{2n-1}\\&&+\varepsilon_1 + 2\varepsilon_2 +\cdots+(n-1)\varepsilon_{n-1}+n\varepsilon_{n}+ n\varepsilon_{n+1})q^{2n-1}\\&&+\varepsilon_2 + 2\varepsilon_3 +\cdots+(n-1)\varepsilon_{n}+n\varepsilon_{n+1})q^{2n}+ \cdots \\&& 
+(2\varepsilon_{n+1}+ \varepsilon_{n})q^{3n-1}+  \varepsilon_{n+1}q^{3n}.
\end{eqnarray*}

From this we get in particular that: $$c_{n+j+1}-c_{n+j}= -\varepsilon_{0}-\varepsilon_{1} - \varepsilon_{j} +\hat\varepsilon_{j+1}+\varepsilon_{j+2}+...+ \varepsilon_{n} + \varepsilon_{n+1} \leq 0$$ and finally the unimodality of $p_{n+2}(q)[n+1]_q[n]_q$.
\end{proof}

\section{Calculations and Examples}\label{calculationsExamples}

\begin{example}
Consider the hedgehog rooted tree $T=114\cdots411$ ($k$-times 4) denoted by $1^{2}4^{k}1^{2}$. See Figure \ref{fig:threekhedge}. We show that for $k=2, 3, 4, 5, 6$ the plucking polynomial is not unimodal. Concrete calculations are given below.

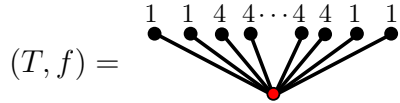
\begin{figure}[ht]
    \centering
    $$(T, f) = \vcenter{\hbox{
\begin{overpic}[scale = .8]{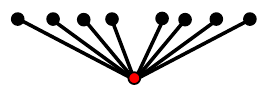}
\put(3,34){\fontsize{10}{10}$1$}
\put(17,34){\fontsize{10}{10}$1$}
\put(29,34){\fontsize{10}{10}$4$}
\put(40,34){\fontsize{10}{10}$4$}
\put(46,34){\fontsize{10}{10}$\cdots$}
\put(59,34){\fontsize{10}{10}$4$}
\put(69,34){\fontsize{10}{10}$4$}
\put(80,34){\fontsize{10}{10}$1$}
\put(94,34){\fontsize{10}{10}$1$}
\end{overpic} }}$$
    \caption{Hedgehog rooted tree $T=114\cdots411$.}
    \label{fig:threekhedge}
\end{figure}

\begin{itemize} 
 \item [(a)] $T=1^24^21^2$:
\begin{eqnarray*}
Q(T, f) &=& q^{13} + 5q^{12} + 12q^{11} + 18q^{10} + 19q^9 + 17q^8\\
 && + 17q^7 + 19q^6 + 18q^5 + 12q^4 + 5q^3 + q^2\\
 &=&
  q^2 (1 + q)^3 (1 + q + q^2)^2 (1 + q^4) \\
 &=&q^2[2]^{2}_q[3]_q[3]_q!(1 + q^4).
\end{eqnarray*}
 \item [(b)] $T=1^24^31^2$:
 \begin{eqnarray*}
Q(T, f) &=& q^{18} + 6q^{17} + 18q^{16} + 36q^{15} + 53q^{14} + 61q^{13} + 59q^{12} + 54q^{11} \\
&&+ 54q^{10} + 59q^9 + 61q^8 + 53q^7 + 36q^6 + 18q^5 + 6q^4 + q^3 \\
  &=&  q^3 (1 + q)^5 (1 + q^2) (1 + q + q^2)^2 (1 - q + q^2 - q^3 + q^4) \\
  &=& q^{3}[2]^{2}_{q}[3]_{q}[4]_{q}!(1+q^5).
\end{eqnarray*}
 \item [(c)] $T=1^24^41^2$:
 \begin{eqnarray*}
       Q(T,f) &=& q^{24} + 7q^{23} + 25q^{22} + 61q^{21} + 114q^{20} + 173q^{19} + 221q^{18} \\
       &&+ 245q^{17} + 245q^{16} + 234q^{15} + 228q^{14} + 234q^{13} + 245q^{12} \\
       &&+ 245q^{11} + 221q^{10} + 173q^9 + 114q^8 + 61q^7 + 25q^6 + 7q^5 + q^4 \\
 &=& q^4 (1 + q)^4 (1 + q^2)^2 (1 + q + q^2)^2 (1 - q^2 + q^4) (1 + q + 
   q^2 + q^3 + q^4)\\
   &=& q^4[2]^{2}_{q}[3]_{q}[5]_{q}!(1+q^6).
   \end{eqnarray*} 

\item [(d)] $T=1^24^51^2$:
\begin{eqnarray*}
	Q(T, f) &=& q^5 + 8 q^6 + 33 q^7 + 94 q^8 + 208 q^9 + 381 q^{10} + 600 q^{11} + 
 832 q^{12} + 1034 q^{13}\\
 && + 1171 q^{14} + 1232 q^{15} + 1234 q^{16} + 
 1212 q^{17} + 1200 q^{18} + 1212 q^{19} + 1234 q^{20} \\
 && + 1232 q^{21}+ 
 1171 q^{22} + 1034 q^{23} + 832 q^{24} + 600 q^{25} + 381 q^{26} + 208 q^{27} + 
 94 q^{28}\\
 && + 33 q^{29} + 8 q^{30} + q^{31} \\
 &=&q^5[2]^{2}_{q}[3]_{q}[6]_{q}!(1+q^7).
	\end{eqnarray*}
 \item [(e)] $T=1^24^61^2$:
    \begin{eqnarray*}
       Q(T, f) &=& q^{39} + 9q^{38} + 42q^{37} + 136q^{36} + 344q^{35} + 725q^{34} + 1325q^{33} + 2155q^{32} + 3174q^{31}\\&& + 4287q^{30} + 5364q^{29} + 6276q^{28} + 6934q^{27} + 7315q^{26} + 7465q^{25} \\&& + 7477q^{24} + 7451q^{23} + 7451q^{22} + 7477q^{21} + 7465q^{20} + 7315q^{19} \\&&+ 6934q^{18} + 6276q^{17}+ 5364q^{16} + 4287q^{15}  + 3174q^{14} + 2155q^{13} + 1325q^{12} + 725q^{11} \\&&+ 344q^{10} + 136q^9 + 42q^8 + 9q^7 + q^6\\
 &=&q^6[2]^{2}_{q}[3]_{q}[7]_{q}!(1+q^8).
    \end{eqnarray*}


\end{itemize}

We have checked that the polynomial is unimodal for $k=1, 7, 8, 9, 10$ and we conjecture that is unimodal for any $k>6$. In particular for $k=7$ we get:

\begin{eqnarray*}
    Q(T, f) &=& q^{48} + 10q^{47} + 52q^{46} + 188q^{45} + 532q^{44} + 1257q^{43} + 2582q^{42} + 4737q^{41} \\
    &&+ 7909q^{40} + 12179q^{39} + 17468q^{38} + 23514q^{37} + 29896q^{36} + 36105q^{35} \\
    &&+ 41645q^{34} + 46137q^{33} + 49397q^{32} + 51464q^{31} + 52566q^{30} + 53031q^{29} \\
    &&+ 53170q^{28} + 53170q^{27} + 53031q^{26} + 52566q^{25} + 51464q^{24} + 49397q^{23} \\
    &&+ 46137q^{22} + 41645q^{21} + 36105q^{20} + 29896q^{19} + 23514q^{18} + 17468q^{17} \\
    &&+ 12179q^{16} + 7909q^{15} + 4737q^{14} + 2582q^{13} + 1257q^{12} + 532q^{11} + 188q^{10} \\
    &&+ 52q^9 + 10q^8 + q^7 \\
    &=& q^7[2]^{2}_{q}[3]_{q}[8]_{q}!(1+q^9).
\end{eqnarray*}
    
\end{example}

We can write in general the following proposition.
\begin{proposition}
Consider the hedgehog tree $T=114\cdots411$ ($k$-times 4).
Its plucking polynomial is given by the formula:
$$Q(T,f)= q^k[k+1]_q! (1+q)^2(1+q+q^2)(1+q^{k+2}).$$
\end{proposition}
\begin{proof} 
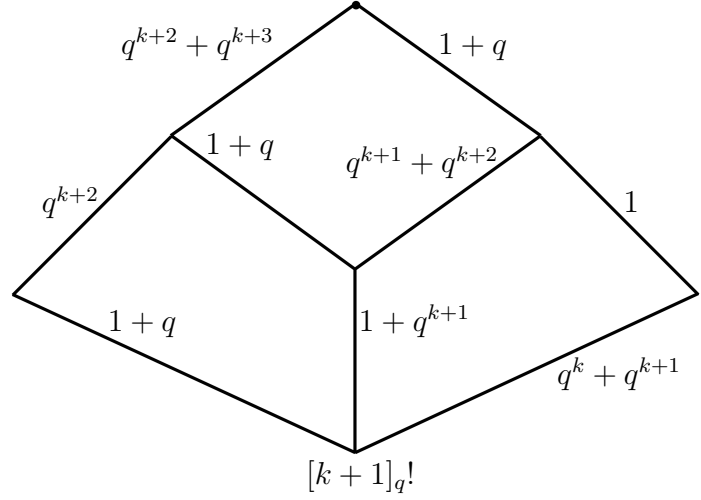
\begin{figure}[ht]
\centering
$$\vcenter{\hbox{
\begin{overpic}[scale = .7]{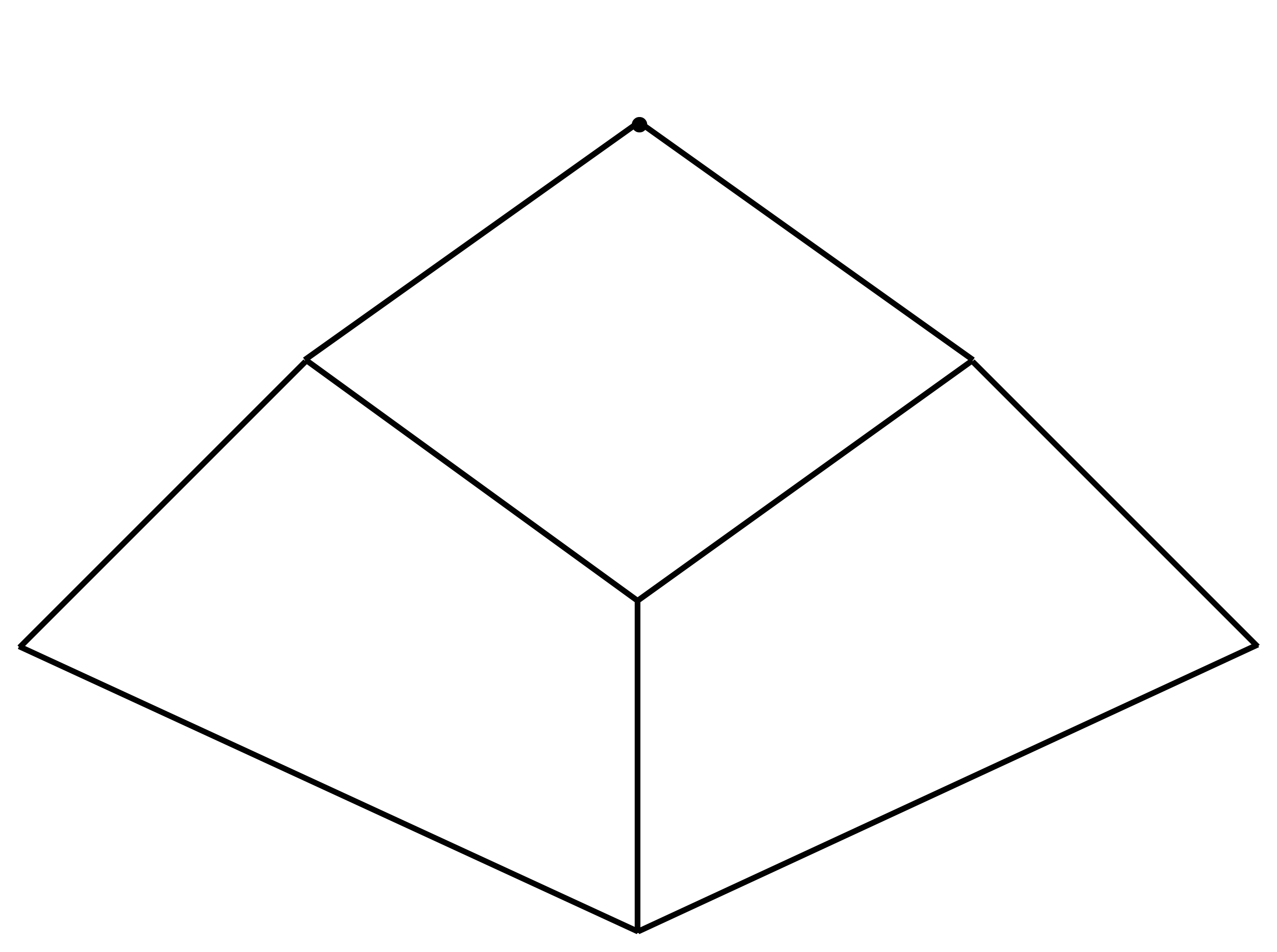}
\put(165,155){$1+q$}
\put(45,155){$q^{k+2}+q^{k+3}$}
\put(15,95){$q^{k+2}$}
\put(77,117){$1+q$}
\put(130,110){$q^{k+1}+q^{k+2}$}
\put(235,95){$1$}
\put(135,50){$1+q^{k+1}$}
\put(40,50){$1+q$}
\put(210,30){$q^{k}+q^{k+1}$}
\put(115,-8){$[k+1]_{q}!$}
\end{overpic} }}$$
\caption{Computation of $Q(T,f)$ for the tree $T=114\cdots411$ ($k$-times 4).}
	\label{proof14k1}
\end{figure}
\begin{eqnarray*}
Q(T, f) &=& (q^{k+3} + q^{k+2})(q^{k+2}(q+1)[k+1]_q!+(1+q)(q^{k+1}+1)[k+1]_q!)\\
&&+ (q+1)((q^{k+2}+q^{k+1})(q^{k+1}+1) [k+1]_q! + (q^{k+1}+q^{k})[k+1]_q!) \\
&=& [k+1]_{q}!(1+q)^{2}(q^{k}+q^{2k+4}+(1+q^{k+1})(q^{k+2}+q^{k+1})) \\&=& [k+1]_{q}!(1+q)^{2}(q^k)(1+q^{k+4}+q^{2}+q+q^{k+3}+q^{k+2})\\&=&q^{k}[k+1]_{q}!(1+q)^{2}(1+q+q^2)(1+q^{k+2}).
      \end{eqnarray*}


\end{proof}

\begin{example}
    Consider the hedgehog rooted tree denoted by $1^a3^k1^b$, that is 
 $a$-times 1, $k$-times 3, and $b$ times 1. Then $$Q(T_{1^a3^k1^b},f)=$$
 $$[k+a+b-2]_q!\bigg((q^{k+b-1}+q^{k+b})[a]_q[b]_q+ q^{2k+2b}[a]_q[a-1]_q + [b]_q[b-1]_q\bigg).$$
 The proof of the formula for this example is explained in Figure \ref{Calculation-1a3k1b}. 
\end{example}

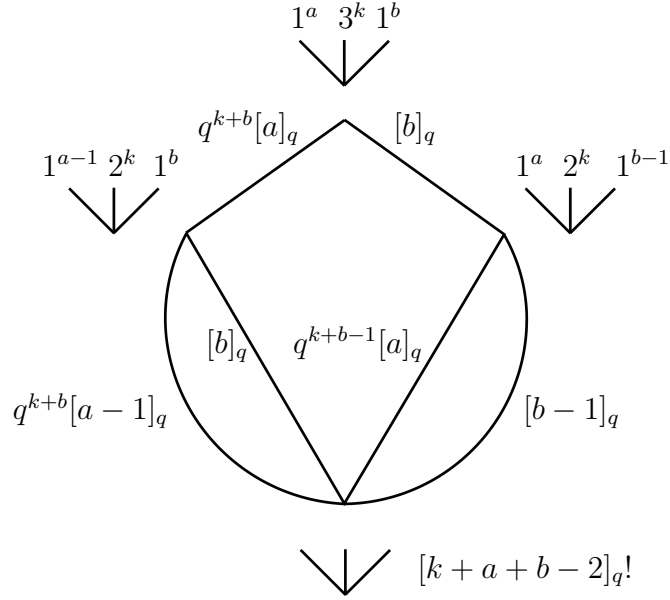
\begin{figure}[ht]
\centering
$$\vcenter{\hbox{
\begin{overpic}[scale = .6]{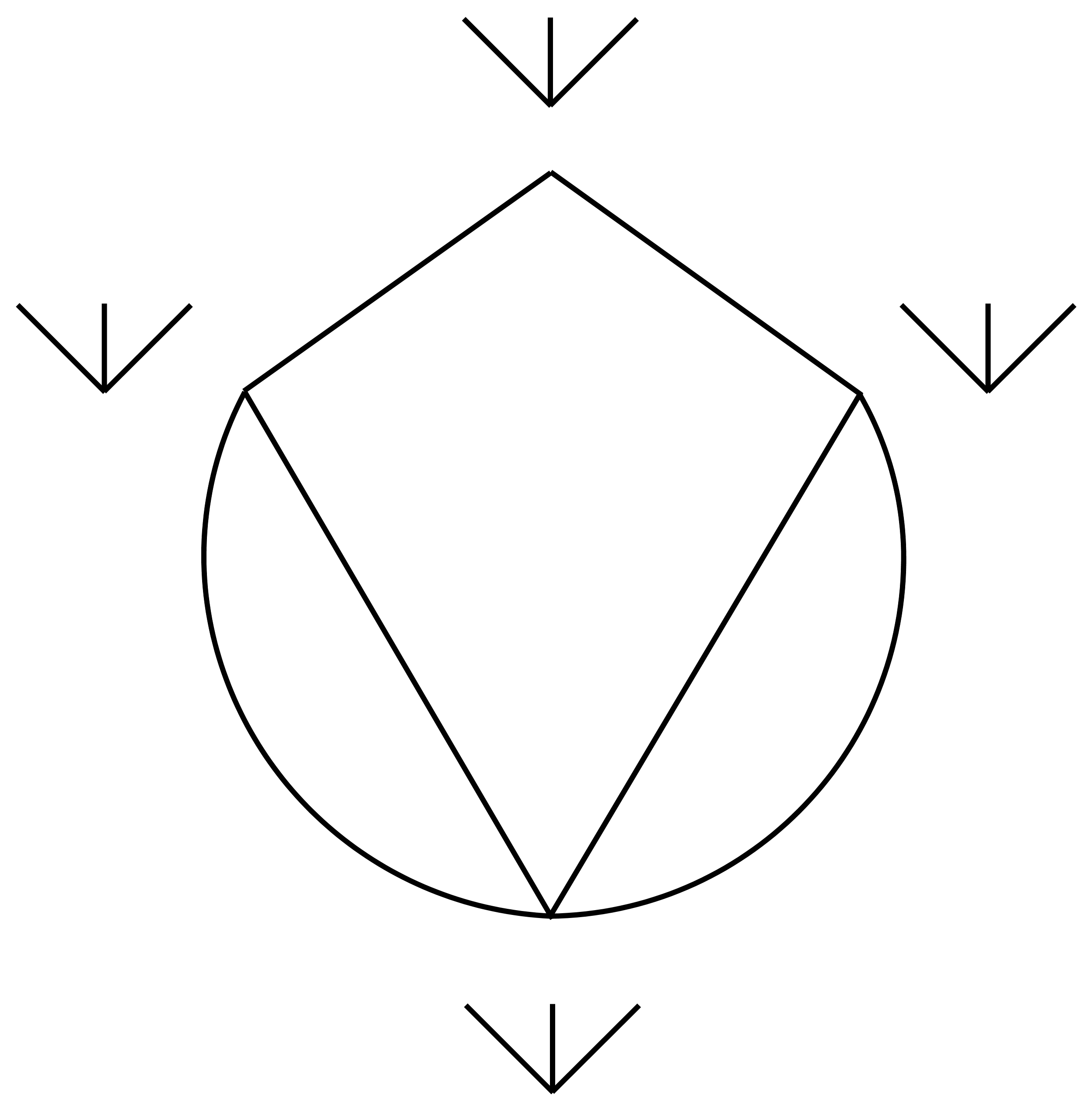}
\put(87,218){$1^a$}
\put(105,218){$3^k$}
\put(119,218){$1^b$}
\put(172,162){$1^a$}
\put(190,162){$2^k$}
\put(210,162){$1^{b-1}$}
\put(-7,162){$1^{a-1}$}
\put(18,162){$2^{k}$}
\put(35,162){$1^{b}$}
\put(52,177){$q^{k+b}[a]_{q}$}
\put(126,177){$[b]_{q}$}
\put(-18,70){$q^{k+b}[a-1]_{q}$}
\put(55,95){$[b]_{q}$}
\put(88,95){$q^{k+b-1}[a]_{q}$}
\put(175,70){$[b-1]_{q}$}

\put(135,10){$[k+a+b-2]_{q}!$}

\end{overpic} }}$$
\caption{Calculation graph for $1^a3^k1^b$. On the bottom we have hedgehog rooted tree $1^{k+a+b-2}$ with the plucking polynomial $[k+a+b-2]_q!$.}
\label{Calculation-1a3k1b}
\end{figure}

\begin{remark}
Another example of a hedgehog rooted tree with delay function and not unimodal plucking polynomial was given by Wojciech Garstka \cite{Ga}. It is a hedgehog rooted tree of type $21412$, see Figure \ref{Garstka1}, with the plucking polynomial
$$q^2(1+q)^2(1+2q+2q^3+q^4)=q^2+4q^3+5q^4+4q^5+5q^6+4q^7+q^8.$$
\end{remark}

\begin{figure}[h]
\centering
$$\vcenter{\hbox{\begin{overpic}[scale = .8]{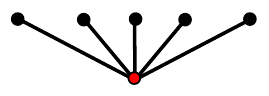}
\put(3,35){$2$}
\put(29,35){$1$}
\put(49,35){$4$}
\put(69,35){$1$}
\put(94,35){$2$}
\end{overpic}}}$$
\caption{The hedgehog rooted tree of type $21412$.}
\label{Garstka1}
\end{figure}
\ \\
Furthermore, type $214412$ is also not unimodal. However, types $1214121$, $112141211$, $2114112$ and $211141112$ are  unimodal. 

\section{Future work}\label{FutureWork}
The plucking polynomial is of particular interest to knot theorists given its connections to the Kauffman bracket skein relations. Experimental calculation suggests the unimodality of the plucking polynomial with delay function values $1$ or $2$. Thus, we propose the following conjecture.

\begin{conjecture}\label{unimodconjecture}(Unimodality Conjecture)
    Let $T = T_{1, 1, \cdots, 1}$ be the Hedgehog (star) rooted tree with n rays and let $f$ be a delay function with values $1$ or $2$. Then its plucking polynomial with delay function $Q(T, f)$, is unimodal. 
\end{conjecture}

The conjecture has been verified up to 8 leaves. 

\ 

One could also investigate whether the plucking polynomial with delay function having values $1, 2,$ or $3$ is always unimodal. As of now, there is not enough data to propose it as a conjecture. 

\ 

The following open problem relates the plucking polynomial with knot theory via the Kauffman bracket of the lattice crossings (see \cite{DP}).
\begin{openproblem}
    Is the plucking polynomial with delay function unimodal when the delay function is anti-unimodal?
\end{openproblem}

\section*{Acknowledgements}
The first author acknowledges the support by the Australian Research Council through grant DP210103136. The third author acknowledges the support of the National Science Foundation through Grant DMS-2212736. The fourth author was partially supported by the Simons Collaboration Grant 637794.



\begin{thebibliography}{999999999}
	
\bibitem[CMPWY1]{CMPWY1} Z. Cheng, S. Mukherjee, J. H. Przytycki, X. Wang and S. Y. Yang, Realization of plucking polynomials,  {\it Journal
of Knot Theory and Its Ramifications}, 26 (2017), 1740016 (9 pages).

\bibitem[CMPWY2]{CMPWY2} Z. Cheng, S. Mukherjee, J. H. Przytycki, X. Wang, S. Y. Yang, Strict unimodality of $q$-polynomials of rooted trees, {\it Journal of Knot Theory and Its Ramifications}, 27(7), June 2018, 1841009 (19 pages). e-print: \href{https://arxiv.org/abs/1601.03465}{arXiv:1601.03465 [math.CO] }.

\bibitem[CMPWY3]{CMPWY3} Z. Cheng, S. Mukherjee, J. H. Przytycki, X. Wang, S. Y. Yang,
Rooted trees with the same plucking polynomials,
    {\it Osaka Journal of Mathematics}, 56, 2019, 661--674. e-print: \href{https://arxiv.org/abs/1702.02004}{arXiv:1702.02004 [math.GT]}.

\bibitem[DLP]{DLP} M. K. Dabkowski, C. Li, J. H. Przytycki, Catalan states of lattice crossing,
\textit{Topology and its Applications}, 182, March, 2015, 1-15. e-print: \href{https://arxiv.org/abs/1409.4065}{arXiv:1409.4065 [math.GT]}.


\bibitem[DP]{DP} M. K. Dabkowski, J. H. Przytycki, Catalan states of lattice crossing: an application of the plucking polynomial, \textit{Topology and its applications}, 254, 12-28, 2019. e-print: \href{https://arxiv.org/pdf/1711.05328.pdf}{arXiv:1711.05328 [math.GT]}.



\bibitem[Gar]{Ga} W. Garstka, private correspondence with J. H. Przytycki. (Plucking polynomial computational project) December 2022.

\bibitem[Prz]{Prz} J.H. Przytycki, q-polynomial invariant of rooted trees; \textit{Arnold Mathematical Journal}, 2(4), 449-461, 2016. e-print: \href{https://arxiv.org/abs/1512.03080}{arXiv:1512.03080 [math.CO]}.

\bibitem[PBIMW]{PBIMW} J. H. Przytycki, R. P. Bakshi, D. Ibarra, G. Montoya-Vega,  D. Weeks,  Lectures in Knot Theory: An Exploration of Contemporary Topics, {\it Springer Universitext}, 
to be published, January 2024.
	
\bibitem[Sta]{Sta}
R.~P.~Stanley, Log-concave and unimodal sequences in algebra, combinatorics, and geometry,
{\it Ann. New York Sci.}, 576, New York Acad. Sci., New York, 1989, 500-535.

\bibitem[Win]{Win}
 A. Wintner, Asymptotic Distributions and Infinite Convolutions, \textit{Edvards
Brothers}, Ann Arbor, Michigan, 1938.

	
	\end{thebibliography}
	\end{document}